\documentclass[draft]{publmathdeb}
\usepackage{graphicx} 
\usepackage[utf8]{inputenc}
\usepackage{amsmath}
\usepackage{amscd}
\usepackage{amsthm}
\usepackage{ upgreek }
\usepackage{amssymb}
\usepackage{hyperref}
\usepackage{enumitem}
\usepackage{url}
\usepackage{times}
\newtheorem{theorem}{Theorem}
\newtheorem{lemma}{Lemma}[section]
\newtheorem{remark}{Remark}
\numberwithin{equation}{section}
\def\ud{{\rm d}}

\author{K. Venkatasubbareddy}
\curraddr{School of Mathematics and Statistics\\
University of Hyderabad\\
Hyderabad\\
India-500046}

\email{20mmpp02@uohyd.ac.in}

\author{A. Sankaranarayanan}
\curraddr{School of Mathematics and Statistics\\
University of Hyderabad\\
Hyderabad\\
India-500046}
\email{sank@uohyd.ac.in}

\title[On certain kernel functions and shifted convolution sums of the Fourier coefficients]
      {On certain kernel functions and shifted convolution sums of the Fourier coefficients}

\keywords{Fourier coefficients of automorphic forms, Dirichlet series, Perron's formula, Dirichlet character.}

\subjclass{Primary 11F30, 11F66.}

\submitted{\today}

\begin{document}

\begin{abstract}
 We study the behavior of the shifted convolution sum involving fourth power of the Fourier coefficients of holomorphic cusp forms with a weight function to be the $k$-full kernel function for any fixed integer $k\geq2$.
\end{abstract}

\maketitle

\section{Introduction}
Let $k\geq 2$ be any fixed integer. Then any integer $n\geq 1$ can be uniquely decomposed as $n=q(n)k(n),\ (q(n),\ k(n))=1$ where $q(n)$ is $k$-free and $k(n)$ is $k$-full ($k(n)$ is $k$-full if $p^k\mid k(n)$ whenever $p\mid k(n)$). By Ivi{\'c} and Tenenbaum \cite{15}, a non negative integer valued function $a(n)$ is called $k$-full kernel function if $a(n)=a(k(n))$ for all $n\geq 1$ and $a(n)\ll n^\epsilon$ for any $\epsilon>0.$ It is noted to be that $k$-full kernel functions are not necessarily multiplicative.\\
For an even integer $\kappa\geq 2$, we denote $H_\kappa$ be the set of all normalized primitive Hecke eigencusp forms $f(z)$ for the full modular group $\Gamma=SL(2, \ \mathbb{Z})$. It is known that $f(z)$ has the Fourier expansion at the cusp $\infty$
\begin{equation*}
    f(z)=\sum_{n=1}^\infty \lambda_f(n) n^{\frac{\kappa-1}{2}} e^{2\pi inz} \textit{ for $\Im(z)>0$},
\end{equation*}
where $\lambda_f(n)$ is real and satisfies the multiplicative condition 
\begin{equation*}
    \lambda_f(m)\lambda_f(n)=\sum_{d\mid(m,\ n)} \lambda_f\left(\frac{mn}{d^2}\right)
\end{equation*}
for all $m,\ n\geq 1$.
By Deligne \cite{5}, \cite{6}, there exists two complex numbers $\alpha_f(p)$ and $\beta_f(p)$ such that 
\begin{equation}
    \lambda_f(p)=\alpha_f(p)+\beta_f(p) \textit{ and } \alpha_f(p)\beta_f(p)=\mid\alpha_f(p)\mid=\mid\beta_f(p)\mid=1\label{Eq1.1}
\end{equation}
 and also for holomorphic cusp forms, Deligne \cite{5} proved the Ramanujan- Petersson conjecture
\begin{equation*}
    \mid\lambda_f(n)\mid\leq d(n).
\end{equation*}
In the history, many authors investigated about shifted convolution sums with $GL(2)$ Fourier coefficients (see \cite{2}, \cite{11} and \cite{17}), Erd\H{o}s and 
Ivi{\'c} \cite{7} studied convolution sums with divisor function $d(n)$ and $\omega(n)$, obtained the following asymptotic relations
\begin{align*}
    \sum_{n\leq x}a(n)d(n+1)=&C_1x\log x+C_2x+O\left(x^{\frac{8}{9}+\epsilon}\right)\\
    \sum_{n\leq x}a(n)\omega(n+1)=&D_1x\log \log x+D_2x+O\left(\frac{x}{\log x}\right),
\end{align*}
where $C_1,\ D_1>0$ and $C_2,\ D_2$ are constants can be evaluated explicitly.\\
In the paper \cite{9}, Guangshi L{\"u} and Dan Wang  investigated the shifted convolution sums of squares of Fourier coefficients with square full kernel function $a^*(n)$ and obtained an asymptotic formula for the sum
\begin{equation*}
\sum_{n\leq x}a^*(n)\lambda_f^2(n+1).
\end{equation*}
In this paper we are interested in the shifted convolution sum of the fourth power of the Fourier coefficients with $k$-full kernel function $a(n)$ for any fixed integer $k\geq 2$. More precisely we study the sum 
$\displaystyle{\sum_{n\leq x}a(n)\lambda_f^4(n+1)}$.

We prove:
\begin{theorem}
Let $f\in H_\kappa$ and $q\geq 100$ be any integer. Then for any $\epsilon>0$ and $q\ll x^{\frac{23}{181}-\epsilon}$, we have
\begin{equation*}
    \sum_{\substack{n\leq x+1\\n\equiv 1(q)}}\lambda_f^4(n)=c_1 x\log x\frac{\phi(q)}{q^2}+O\left(\frac{x^{\frac{158}{181}+\epsilon}q^{1+\epsilon}}{\phi(q)}\right)
\end{equation*}
uniformly.
\end{theorem}

\begin{theorem}
Let $f\in H_\kappa $ and $q\geq 100$ be prime. Then for any $\epsilon>0$ and $q\ll x^{\frac{3}{23}-\epsilon}$, we have
\begin{equation*}
    \sum_{\substack{n\leq x+1\\n\equiv 1(q)}}\lambda_f^4(n)=c_1 x\log x\frac{\phi(q)}{q^2}+O\left(\frac{x^{\frac{20}{23}+\epsilon}q^{1+\epsilon}}{\phi(q)}\right)
\end{equation*}
uniformly.
\end{theorem}
In both the theorems 1 and 2, $c_1$ is the same constant that can be evaluated explicitly.

As an application of theorem 1 we obtain:
\begin{theorem}
For any integer $k\geq 2$, let $a(n)$ be the $k$-full kernel function and $f\in H_\kappa$. Then for any $\epsilon>0$, we have
\begin{equation*}
    \sum_{n\leq x}a(n)\lambda_f^4(n+1)=c_2x\log x+O\left(x^{\frac{520k+23}{543k}+\epsilon}\right).\qed
\end{equation*}
where $c_2$ is a constant can be evaluated explicitly.
\end{theorem}

\begin{remark}
Throughout this paper $\epsilon$ is any small positive constant and all the implied constants might depend upon $k$ and $\epsilon$.
\end{remark}

\section{Preliminaries and lemmas}
For $f\in H_\kappa$, the Hecke $L$-function attached to $f$ is defined as
\begin{equation*}
    L(s,\ f)=\sum_{n=1}^\infty\frac{\lambda_f(n)}{n^s}=\prod_p\left(1-\frac{\alpha_f(p)}{p^s}\right)^{-1}\left(1-\frac{\beta_f(p)}{p^s}\right)^{-1}
\end{equation*}
which converges absolutely for $\Re(s)>1$, where $\alpha_f(p)$ and $\beta_f(p)$ satisfies \eqref{Eq1.1}.\\
The $2^{\textit{nd}}$ symmetric power $L$-function attached to $f$ defined as 
\begin{align*}
    L(s,\ {\rm{sym}}^2f)=&\sum_{n=1}^\infty \frac{\lambda_{{\rm{sym}}^2f}(n)}{n^s}\\
    =&
    \prod_p\prod_{j=0}^2\left(1-\frac{\alpha_f^{2-j}(p)\beta_f^j(p)}{p^s}\right)^{-1}
    \end{align*}
which converges absolutely for $\Re(s)>1$.\\
For a Dirichlet character modulo $q$ the twisted $2^{\textit{nd}}$ power $L$-function attached to $f$ is defined as 
\begin{align*}
    L(s,\ {\rm{sym}}^2f\otimes\chi)=&\sum_{n=1}^\infty \frac{\lambda_{{\rm{sym}}^2f}(n)\chi(n)}{n^s}\\
    =&\prod_p\prod_{0\leq j\leq 2
    }\left(1-\frac{\alpha_f^{2-2j}(p)\chi(p)}{p^s}\right)^{-1}
\end{align*}
    for $\Re(s)>1$ and $L(s,\ {\rm{sym}}^2f\otimes \chi)$ is of degree 3, similarly the twisted $4^{\textit{th}}$ power $L$-function attached to $f$ is defined as 
\begin{align*}
    L(s,\ {\rm{sym}}^4f\otimes\chi)=&\sum_{n=1}^\infty \frac{\lambda_{{\rm{sym}}^4f}(n)\chi(n)}{n^s}\\
    =&\prod_p\prod_{0\leq j\leq 4
    }\left(1-\frac{\alpha_f^{4-2j}(p)\chi(p)}{p^s}\right)^{-1}
\end{align*}
    for $\Re(s)>1$ and $L(s,\ {\rm{sym}}^4f\otimes\chi)$ is of degree 5.\\
    Now we define 
    \begin{equation*}
        F_f(s,\ \chi):=\sum_{n=1}^\infty\frac{\lambda_f^4(n)\chi(n)}{n^s}
    \end{equation*}
for $\Re(s)>1$ then by the multiplicative property of $\lambda_f(n)$ and $\chi(n)$, we have
\begin{equation*}
    F_f(s,\ \chi)=\prod_p\left(1+\frac{\lambda_f^4(p)\chi(p)}{p^s}+\frac{\lambda_f^4(p^2)\chi(p^2)}{p^{2s}}+\hdots\right).
\end{equation*}
Corresponding to each Dirichlet character $\chi$ modulo $q$ there exists a conductor $q_1$, the smallest divisor of $q$ such that $\chi=\chi_0\chi^*$, where $\chi_0$ is principal character modulo $q$ and $\chi^*$ is a Dirichlet character modulo $q_1$. For some characters we have $q_1=q$, such characters are called primitive characters.

\begin{lemma}
For $\Re(s)>1$, define
\begin{equation*}
    F_f(s,\ \chi):=\sum_{n=1}^\infty \frac{\lambda_f^4(n)\chi(n)}{n^s}
\end{equation*}
where $\chi$ is a Dirichlet character modulo $q$. Then we have 
\begin{equation}
    F_f(s,\ \chi)=L^2(s,\ \chi)L^3(s,\ {\rm{sym}^2}f\otimes \chi)L(s,\ \rm{sym}^4f\otimes \chi)U(s)\label{Eq2.1}
\end{equation}
where $U(s)$ is some Dirichlet series which converges absolutely in $\Re(s)\geq \frac{1}{2} +\epsilon$.
\end{lemma}

\begin{proof}
From (1.1), we have
\begin{align*}
    \lambda_f^4(p)=&(\alpha_f(p)+\beta_f(p))^4\\
    =&\alpha_f^4(p)+\beta_f^4(p)+4\alpha_f^3(p)\beta_f(p)+4\alpha_f(p)\beta_f^3(p)+6\alpha_f^2(p)\beta_f^2(p)\\
    =&\alpha_f^4(p)+\beta_f^4(p)+4\alpha_f^2(p)+4\beta_f^2(p)+6\\
    =&(\alpha_f^4(p)+\beta_f^4(p)+\alpha_f^2(p)+\beta_f^2(p)+1)+3(\alpha_f^2(p)+\beta_f^2(p)+1)+2.
\end{align*}
So 
\begin{align*}
    \lambda_f^4(p)\chi(p)=&(\alpha_f^4(p)+\beta_f^4(p)+\alpha_f^2(p)+\beta_f^2(p)+1)\chi(p)\\&\qquad+3(\alpha_f^2(p)+\beta_f^2(p)+1)\chi(p)+2\chi(p).
\end{align*}
Since $\lambda_f^4(n)$ and $\chi(n)$ are multiplicative functions, by standard arguments the above relation leads us to obtain 
\begin{align*}
    F_f(s,\ \chi)=&\prod_p\left(1+\frac{\lambda_f^4(p)\chi(p)}{p^s}+\frac{\lambda_f^4(p^2)\chi(p^2)}{p^{2s}}+\hdots\right)\\
    =&L^2(s,\ \chi)L^3(s,\ {\rm{sym}^2}f\otimes \chi)L(s,\ \rm{sym}^4f\otimes \chi)U(s),
\end{align*}
where $U(s)$ is some Dirichlet series which converges absolutely in $\displaystyle{\Re(s)\geq\frac{1}{2}+\epsilon}$.
\end{proof}

\begin{lemma}
Let $\chi_0$ be a principal character modulo $q$. Then we have 
\begin{align}
    L(s,\ \chi_0)=&\zeta(s)\prod_{p\mid q}\left(1-\frac{1}{p^s}\right),\label{Eq2.2}\\
     L(s,\ {\rm{sym}^2}f\otimes\chi_0)=&L(s,\ {\rm{sym}^2}f)\prod_{p\mid q}\prod_{0\leq j\leq 2}\left(1-\frac{\alpha_f^{2-2j}(p)}{p^s}\right),\label{Eq2.3}\\
     L(s,\ {\rm{sym}^4}f\otimes\chi_0)=&L(s,\ {\rm{sym}^4}f)\prod_{p\mid q}\prod_{0\leq j\leq 4}\left(1-\frac{\alpha_f^{4-2j}(p)}{p^s}\right).\label{Eq2.4}
\end{align}
\end{lemma}

\begin{proof}
For the principal character $\chi_0$ modulo $q$, we have by definition 
\begin{align*}
    L(s,\ \chi_0)=&\prod_p\left(1-\frac{\chi_0(p)}{p^s}\right)^{-1}\\
    =&\prod_{p\nmid q}\left(1-\frac{1}{p^s}\right)^{-1}\\
    =&\prod_p \left(1-\frac{1}{p^s}\right)^{-1}\prod_{p\mid q} \left(1-\frac{1}{p^s}\right)\\
    =&\zeta(s)\prod_{p\mid q} \left(1-\frac{1}{p^s}\right).
\end{align*}

\begin{align*}
    L(s,\ {\rm{sym}^2}f\otimes\chi_0)=&\prod_p\prod_{0\leq j\leq 2}\left(1-\frac{\alpha_f^{2-2j}(p)\chi_0(p)}{p^s}\right)^{-1}\\
    =&\prod_{p\nmid q}\prod_{0\leq j\leq 2}\left(1-\frac{\alpha_f^{2-2j}(p)}{p^s}\right)^{-1}\\
    =&\prod_p \prod_{0\leq j\leq 2}\left(1-\frac{\alpha_f^{2-2j}(p)}{p^s}\right)^{-1}\prod_{p\mid q}\prod_{0\leq j\leq 2} \left(1-\frac{\alpha_f^{2-2j}(p)}{p^s}\right)\\
    =&L(s,\ {\rm{sym}^2}f)\prod_{p\mid q}\prod_{0\leq j\leq 2} \left(1-\frac{\alpha_f^{2-2j}(p)}{p^s}\right).
\end{align*}
The other equality follows in a similar manner.
\end{proof}

\begin{lemma}
Let $\chi$ be a non-primitive character modulo $q$ and $\chi^*$ be primitive character modulo $q_1(\neq q)$, induced by $\chi$. Then we have
\begin{align}
    L(s,\ \chi)=&L(s,\ \chi^*)\prod_{\substack{p\mid q\\p\nmid q_1}}\left(1-\frac{\chi^*(p)}{p^s}\right),\label{Eq2.5}\\
    L(s,\ {\rm{sym}}^2f\otimes\chi)=&L(s,\ {\rm{sym}}^2f\otimes\chi^*)\prod_{\substack{p\mid q\\p\nmid q_1\\0\leq j\leq 2}}\left(1-\frac{\alpha_p^{2-2j}\chi^*(p)}{p^s}\right),\label{Eq2.6}\\
    L(s,\ {\rm{sym}}^4f\otimes\chi)=&L(s,\ {\rm{sym}}^4f\otimes\chi^*)\prod_{\substack{p\mid q\\p\nmid q_1\\0\leq j\leq 4}}\left(1-\frac{\alpha_p^{4-2j}\chi^*(p)}{p^s}\right),\label{Eq2.7}
    \end{align}
    
    \begin{align}
    \prod_{\substack{p\mid q\\p\nmid q_1}}\left(1-\frac{\chi^*(p)}{p^s}\right)\ll& q^\epsilon\ \textit{ for } \frac{1}{2}+\epsilon<\Re(s)<1+\epsilon,\label{Eq2.8}\\
    \prod_{\substack{p\mid q\\p\nmid q_1\\0\leq j\leq 2}}\left(1-\frac{\alpha^{2-2j}_p\chi^*(p)}{p^s}\right)\ll& q^\epsilon\ \textit{ for } \frac{1}{2}+\epsilon<\Re(s)<1+\epsilon,\label{Eq2.9}\\
    \textit{ and }\prod_{\substack{p\mid q\\p\nmid q_1\\0\leq j\leq 4}}\left(1-\frac{\alpha^{4-2j}_p\chi^*(p)}{p^s}\right)\ll& q^\epsilon\ \textit{ for } \frac{1}{2}+\epsilon<\Re(s)<1+\epsilon.\label{Eq2.10}
\end{align}
\end{lemma}

\begin{proof}
We have $\chi^*$ is induced by $\chi$,\\
i.e., $\chi=\chi_0\chi^*$, where $\chi_0$ is principal character modulo $q$.\\
The first 3 equalities follow from the above observation and the bounds in \eqref{Eq2.8}, \eqref{Eq2.9} and \eqref{Eq2.10} are trivial.
\end{proof}

\begin{lemma}[KR+AS]
For $\frac{1}{2}\leq \sigma\leq 2$, $T$-sufficiently large, there exist a $T^*\in[T,T+T^\frac{1}{3}]$ such that the bound 
\begin{equation*}
    \log\zeta(\sigma+iT^*)\ll (\log\log T^*)^2\ll(\log\log T)^2
\end{equation*}
holds uniformly for $\frac{1}{2}\leq\sigma\leq 2$ and thus we have 
\begin{equation}
    \mid \zeta(\sigma+iT^*)\mid \ll \exp((\log\log T^*)^2)\ll_\epsilon T^\epsilon\label{Eq2.11}
\end{equation}
on the horizontal line with $t=T^*$ uniformly for $\frac{1}{2}\leq \sigma\leq 2.$
\end{lemma}
\begin{proof}
See, Lemma $1$ of \cite{20}.
\end{proof}

\begin{lemma}
For $\frac{1}{2}\leq \sigma\leq 2$ and $T\geq 2$, we have
\begin{equation}
    \int_1^T\mid\zeta^2(\sigma+it)\mid^2\ \ud t\ll T(\log T)^4\label{Eq2.12}
\end{equation}
holds uniformly.
\end{lemma}
\begin{proof}
see \cite[pp. 148]{22}.
\end{proof}

\begin{lemma}
Let $\chi$ be a primitive character modulo $q$. Then for $q\ll T^2$, we have
\begin{equation}
    L(\sigma+iT,\ \chi)\ll (q(1+\mid T\mid))^{\max\{\frac{1}{3}(1-\sigma),\ 0\}+\epsilon}\label{Eq2.13}
\end{equation}
holds uniformly for $\frac{1}{2}\leq\sigma\leq 2$ and $\mid T\mid \geq 1$;
\begin{equation}
    \int_1^T \mid L^2(\sigma+it,\ \chi)\mid^2\ \ud t\ll (qT)^{2(1-\sigma)+\epsilon}\label{Eq2.14}
\end{equation}
uniformly for $\frac{1}{2}\leq \sigma\leq 1+\epsilon$ and $T\geq 1$;\\
More over if $q$ is a prime, for $q\ll T^2$, we have
\begin{equation}
    \int_1^T\mid L(\sigma+it,\ \chi)\mid^{12} \ud t\ll q^{4(1-\sigma)}T^{3-2\sigma +\epsilon}\label{Eq2.15}
\end{equation}
for $\frac{1}{2}\leq \sigma\leq 2$ and $\mid T\mid\geq1$.
\end{lemma}
\begin{proof}
The results \eqref{Eq2.13} and \eqref{Eq2.14} follows from D. R. Heath-Brown \cite{10} and Perelli \cite{19} respectively, the result \eqref{Eq2.15} follows from the Phragm{\'e}n-Lindel{\"o}f principle and Motohashi \cite{18}.
\end{proof}

\begin{lemma}
Let $f\in H_\kappa$ and $\chi$ be a primitive character modulo $q$. Then for $q\ll T^2$, we have
\begin{equation}
    L(\sigma+iT,\ {\rm{sym}}^2f)\ll (1+\mid T\mid)^{\max \{\frac{27}{20}(1-\sigma),\ 0\}+\epsilon}\label{Eq2.16}
\end{equation}
and
\begin{equation}
    L(\sigma+iT,\ {\rm{sym}}^2f\otimes\chi)\ll (q(1+\mid T\mid))^{\max\{\frac{67}{46}(1-\sigma),\ 0\}+\epsilon}\label{Eq2.17}
\end{equation}
uniformly for $\frac{1}{2}\leq \sigma\leq 2$ and $\mid T\mid\geq 1$;
\begin{equation}
    \int_1^T\mid L\left(\sigma+it,\ {\rm{sym}}^2f\otimes \chi\right)\mid^4\ud t\ll \left(qT\right)^{6(1-\sigma)+\epsilon}\label{Eq2.18}
\end{equation}
uniformly for $\frac{1}{2}\leq \sigma\leq 1+\epsilon$ and $T\geq 1$.
\end{lemma}
\begin{proof}
Here one can observe that  the $L$-function $L\left(\sigma+it,\ {\rm{sym}}^2f\otimes \chi\right)$ is of degree 3. Now \eqref{Eq2.16} and \eqref{Eq2.17} follows from the Phragm{\'e}n-Lindel{\"o}f principle and the work of Aggarwal \cite{1} and Huang \cite{13} respectively, and the result \eqref{Eq2.18} follows from Perelli \cite{19}.
\end{proof}

\begin{lemma}
Suppose that $\mathfrak{L}(s)$ is a general $L$-function of degree $m$. Then, for any $\epsilon>0$, we have
\begin{equation}
    \int_T^{2T}\mid\mathfrak{L}(\sigma+it)\mid^2\ll T^{\max\{m(1-\sigma),\ 1\}+\epsilon}\label{Eq2.19}
\end{equation}
uniformly for $\frac{1}{2}\leq \sigma\leq 1$ and $T>1$; and 
\begin{equation}
    \mathfrak{L}(\sigma+it)\ll (\mid t\mid+1)^{(m/2)(1-\sigma)+\epsilon}\label{Eq2.20}
\end{equation}
uniformly for $\frac{1}{2}\leq \sigma\leq 1+\epsilon$ and $\mid t\mid>1$.
\end{lemma}

\begin{proof}
The result \eqref{Eq2.19} is due to Perelli \cite{19} and \eqref{Eq2.20} follows from Maximum modulus principle.
\end{proof}

\begin{lemma}
Let $f\in H_\kappa$ and $\chi$ be a primitive character modulo $q$. Then for any $\epsilon>0$, we have
\begin{equation}
    L(\sigma+iT,\ {\rm{sym}}^4f\otimes \chi)\ll (qT)^{\max \{\frac{5}{2}(1-\sigma),\ 0\}+\epsilon}\label{Eq2.21}
\end{equation}
uniformly for $-\epsilon\leq \sigma\leq 1+\epsilon$,
\begin{equation}
    \int_1^T\mid L\left(\sigma+it,\ {\rm{sym}}^4f\otimes \chi\right)\mid^{2}\ud t\ll \left(qT\right)^{5(1-\sigma)+\epsilon}\label{Eq2.22}
\end{equation}
uniformly for $\frac{1}{2}\leq \sigma\leq 1+\epsilon$ and $T\geq 1$.
\end{lemma}
\begin{proof}
One can observe that the $L$-function $L\left(s,\ {\rm{sym}}^4f\otimes \chi\right)$ is of degree 5.
Now the results \eqref{Eq2.21} and \eqref{Eq2.22} follows from Perelli \cite{19}.
\end{proof}

\begin{lemma}
For any $x>1$, we have 
\begin{equation*}
    \sum_{p\leq x}\frac{1}{p}<\log\log x+B+\frac{1}{\log^2x}
\end{equation*} for some $B>0$. 
\end{lemma}
\begin{proof}
The inequality holds with $B=0.261497212847643$, see (3.20) of \cite{21}.
\end{proof}

\begin{lemma}
Let $q>100$ be any integer. Then we have
\begin{equation*}
    \frac{q}{\phi(q)}\leq e^{A+1}\log \log q
\end{equation*}
for some effective $A>0$.
\end{lemma}

\begin{proof}
We have 
\begin{align}
    \frac{q}{\phi(q)}=&\prod_{p\mid q}\left( 1-\frac{1}{p}\right)^{-1}\nonumber\\
    =&\exp\left\{-\sum_{p\mid q}\log \left(1-\frac{1}{p}\right) \right\}\nonumber\\
    \leq &\exp\left\{\sum_{p\mid q}\left(\frac{1}{p}+\frac{\frac{1}{p^2}}{1-\frac{1}{p}}\right)\right\}\nonumber\\
    \leq &\exp\left\{\sum_{p\mid q}\frac{1}{p}+1\right\}.\label{Eq2.23}
\end{align}
Now using above lemma, we have
\begin{align*}
    \sum_{p\mid q}\frac{1}{p}=&\sum_{\substack{p\mid q\\p\leq \log q}}\frac{1}{p}+\sum_{\substack{p\mid q\\p> \log q}}\frac{1}{p}\\
    \leq& \log\log\log q+B+\frac{1}{\log^2q}+\frac{1}{\log q}\omega(q)\\
    \leq &\log\log\log q+B+\frac{1}{\log^2q}+\frac{\log q}{(\log 2)\log q}\\
    \leq& \log \log \log q+A
\end{align*}
for some effective $A>0$ and $\omega(q)$ is the number of distinct prime factors of $q$.\\
Hence from \eqref{Eq2.23}, we have 
\begin{align*}
    \frac{q}{\phi(q)}\leq&\exp\big\{\log \log \log q+A+1\big\}\\
    \leq &e^{A+1}\log \log q,
\end{align*}
for some effective $A>0$.\\
From this one can see
\begin{equation}
    \frac{q}{\phi(q)}\ll \log \log q.\label{Eq2.24}
\end{equation}
\end{proof}

\begin{lemma}
Let $k$ denote the $k$-full numbers and $x$ be large. Then we have
\begin{equation}
    \sum_{k\leq x}1 \sim K x^\frac{1}{k}\qquad\textit{ and }\qquad\sum_{k>x}\frac{1}{k}\ll x^{\frac{1}{k}-1}\label{Eq2.25}
\end{equation} 
for some constant $K$ as $x$ tends infinity.
\end{lemma}
\begin{proof}
For the first approximation see \cite[pp. 33]{14} and the bound in $\displaystyle{\sum_{k>x}\frac{1}{k}}$ follows by Riemann-Steiltjes integration.
\end{proof}

\begin{proof}[Proof of theorem 1]
For any Dirichlet character $\chi$ modulo $q$ and by orthogonality we have 
\begin{align}
    \sum_{\substack{n\leq x+1\\n\equiv 1(q)}}\lambda_f^4(n)=&\frac{1}{\phi(q)}\sum_{\chi(q)}\sum_{n\leq x+1}\lambda_f^4(n)\chi(n)\nonumber\\
    =&\frac{1}{\phi(q)}\sum_{n\leq x+1}\lambda_f^4(n)\chi_0(n)+\frac{1}{\phi(q)}\sum_{{\substack{n\leq x+1,\\\chi\neq \chi_0,\\\textit{$\chi$ is primitive}}}}\lambda_f^4(n)\chi(n)\nonumber\\&\qquad+\frac{1}{\phi(q)}\sum_{{\substack{n\leq x+1,\\\chi\neq \chi_0,\\\textit{$\chi$ is non-primitive}}}}\lambda_f^4(n)\chi(n)\nonumber\\
    :=&\textit{$\sum$}_1+\textit{$\sum$}_2+\textit{$\sum$}_3.\label{Eq3.1}
\end{align}
By using \eqref{Eq2.1} and Perron's formula for $F_f(s,\ \chi_0)$, we have
\begin{align*}
    \sum_{n\leq x+1}\lambda_f^4(n)\chi_0(n)=&\frac{1}{2\pi i}\int_{1+\epsilon-iT}^{1+\epsilon+iT}F_f(s,\ \chi_0)\frac{(x+1)^s}{s}\ \ud s+O\left(\frac{x^{1+\epsilon}}{T}\right)\\
    =&\frac{1}{2\pi i}\int_{1+\epsilon-iT}^{1+\epsilon+iT}\zeta^2(s)\prod_{p\mid q}(1-\frac{1}{p^s})^2L^3(s,\ {\rm{sym}}^2f\otimes\chi)\times\\&\qquad L(s,\ {\rm{sym}}^4f\otimes \chi)U(s)\frac{(x+1)^s}{s}\ \ud s
+O\left(\frac{x^{1+\epsilon}}{T}\right)
\end{align*}
where $1\leq T\leq x$ is a parameter to be chosen later.\\
Now we make the special choice $T=T^*$ of Lemma 2.4 (KR+AS) satisfying \eqref{Eq2.11} and by moving the line of integration to $\Re(s)=\frac{1}{2}+\epsilon$, we have by Cauchy theorem 
\begin{align*}
    \sum_{n\leq x+1}\lambda_f^4(n)\chi_0(n)\\=&x\log x\prod_{p\mid q}(1-\frac{1}{p})^2L^3(1,\ {\rm{sym}}^2f\otimes\chi_0)L(1,\ {\rm{sym}}^4f\otimes \chi_0)U(1)\\&-\frac{1}{2\pi i}\int_{C_1}F_f(s,\ \chi_0)\frac{(x+1)^s}{s}\ \ud s+O\left(\frac{x^{1+\epsilon}}{T}\right)
\end{align*}
where the main term is coming from the pole at $s=1$ of order 2 of $\zeta(s)$ and $C_1$ is the curve connecting by the points $1+\epsilon+iT,\ \frac{1}{2}+\epsilon+iT,\ \frac{1}{2}+\epsilon-iT$ and $1+\epsilon-iT$ with straight line segments.\\
Now by Cauchy-Schwarz inequality, we have
\begin{align}
    J_1:=&\int_{\frac{1}{2}+\epsilon-iT}^{\frac{1}{2}+\epsilon+iT}F_f(s,\ \chi_0)\frac{(x+1)^s}{s}\ \ud s\nonumber\\
    =&\int_{\frac{1}{2}+\epsilon-iT}^{\frac{1}{2}+\epsilon+iT}\zeta^2(s)\prod_{p\mid q}\left(1-\frac{1}{p^s}\right)^2\times\nonumber\\
    &L^3(s,\ {\rm{sym}^2}f\otimes \chi_0)L(s,\ {\rm{sym}}^4f\otimes \chi_0)U(s)\frac{(x+1)^s}{s}\ \ud s\nonumber\\
    \ll& \int_{-T}^{T} \mid\zeta^2(1/2+\epsilon+it)L^3(1/2+\epsilon+it,\ {\rm{sym}}^2f)\times\nonumber\\
    &\qquad L(1/2+\epsilon+it,\ {\rm{sym}}^4f)\mid\frac{x^{1/2+\epsilon}}{\mid 1/2+\epsilon+it\mid}\ \ud t\nonumber\\
    \ll& x^{1/2+\epsilon}+\sup_{1\leq T_1\leq T}x^{1/2+\epsilon} T_1^{-1}I_1^{1/2}I_2^{1/2},\label{Eq3.2}
\end{align}
where 
\begin{align*}
    I_1=&\int_{T_1}^{2T_1}\mid\zeta^2(1/2+\epsilon+it)\mid^2\ \ud t,\\
    I_2=&\int_{T_1}^{2T_1}\mid L^3(1/2+\epsilon+it,\ {\rm{sym}^2}f)L(1/2+\epsilon+it,\ {\rm{sym}}^4f)\mid^2\ \ud t.
\end{align*}
By \eqref{Eq2.12}, \eqref{Eq2.16} and \eqref{Eq2.19}, we have
\begin{align*}
    I_1\ll& T_1(\log T_1)^4\\
    I_2\ll & \left(\max_{T_1\leq t\leq 2T_1}\mid L^3(\frac{1}{2}+\epsilon+it,\ {\rm{sym}}^2f)\mid^2\right)\times\\
    &\qquad\left(\int_{T_1}^{2T_1}\mid L(\frac{1}{2}+\epsilon+it,\ {\rm{sym}}^4f)\mid^2\ \ud t\right)\\
    \ll& T_1^{6\times \frac{27}{20}\times\frac{1}{2}+\epsilon+\frac{5}{2}+\epsilon}\\
    \ll& T_1^{\frac{131}
{20}+2\epsilon}.
\end{align*}
By \eqref{Eq3.2}, we have 
\begin{align*}
    J_1=&\int_{\frac{1}{2}+\epsilon-iT}^{\frac{1}{2}+\epsilon+iT}F_f(s,\ \chi_0)\frac{(x+1)^s}{s}\ \ud s\\
    \ll& x^{\frac{1}{2}+\epsilon}(T(\log T)^4)^\frac{1}{2} T^{\frac{131}{20}\times\frac{1}{2}+\epsilon-1}\\
    \ll & x^{\frac{1}{2}+\epsilon}T^{\frac{111}{40}+3\epsilon}.
\end{align*}
Now for horizontal portions by \eqref{Eq2.11}, \eqref{Eq2.16} and \eqref{Eq2.20}, we have
\begin{align*}
    J_2:=&\int_{\frac{1}{2}+\epsilon-iT}^{1+\epsilon+iT}F_f(s,\ \chi_0)\frac{(x+1)^s}{s}\ \ud s\\
    \ll& \int_{\frac{1}{2}+\epsilon}^{1+\epsilon}T^{2\epsilon+3\times \frac{27}{20}(1-\sigma)+\frac{5}{2}(1-\sigma)+\epsilon}\frac{x^\sigma}{T}\ \ud \sigma\\
    \ll& T^{\frac{111}{20}+2\epsilon}\int_{\frac{1}{2}+\epsilon}^{1+\epsilon}\left(\frac{x}{T^{\frac{131}{20}}}\right)^\sigma\ \ud \sigma\\
    \ll& \max_{\frac{1}{2}+\epsilon\leq \sigma\leq 1+\epsilon}\left(\frac{x}{T^{\frac{131}{20}}}\right)^\sigma T^{\frac{111}{20}+2\epsilon}\\
    \ll& \frac{x^{1+\epsilon}}{T}+x^{\frac{1}{2}+\epsilon}T^{\frac{91}{40}+10\epsilon}.
\end{align*}
Hence 
\begin{align}
    \sum_{n\leq x+1}\lambda_f^4(n)\chi_0(n)\nonumber\\=&x\log x\prod_{p\mid q}\left(1-\frac{1}{p}\right)^2L^3(1,\ {\rm{sym}}^2f\otimes\chi_0)\times \nonumber\\&L(1,\ {\rm{sym}}^4f\otimes \chi_0)U(1)+O\left(\frac{x^{1+\epsilon}}{T}\right)+O\left(x^{\frac{1}{2}
     +\epsilon}T^{\frac{111}{40}+\epsilon}\right).\label{Eq3.3}
    \end{align}
Now for $\sum_2$, using Perron's formula we have
\begin{equation*}
    \sum_{{\substack{n\leq x+1\\\chi\neq \chi_0\\\textit{$\chi$ is primitive}}}}\lambda_f^4(n)\chi(n)=\frac{1}{2\pi i} \int_{1+\epsilon-iT}^{1+\epsilon+iT}F_f(s,\ \chi)\frac{(x+1)^s}{s}\ \ud s+O\left(\frac{x^{1+\epsilon}}{T}\right),
\end{equation*}
where $1\leq T\leq x$ is a parameter to be chosen later.\\
By moving the line of integration to $\Re(s)=\frac{1}{2}+\epsilon$, we have by Cauchy's theorem
\begin{equation*}
    \sum_{{\substack{n\leq x+1,\\\chi\neq \chi_0,\\\textit{$\chi$ is primitive}}}}\lambda_f^4(n)\chi(n)=-\frac{1}{2\pi i} \int_{C_2}F_f(s,\ \chi)\frac{(x+1)^s}{s}\ \ud s+O\left(\frac{x^{1+\epsilon}}{T}\right)
\end{equation*}
where $C_2$ is the curve joining the points $1+\epsilon+iT,\ \frac{1}{2}+\epsilon+iT,\ \frac{1}{2}+\epsilon-iT$ and $1+\epsilon-iT$ through straight line segments.
For vertical portions using \eqref{Eq2.1} and Cauchy-Schwarz inequality, we have  
\begin{align}
    J_3:=&\int_{1/2+\epsilon-iT}^{1/2+\epsilon+iT}F_f(s,\ \chi)\frac{(x+1)^s}{s}\ \ud s\nonumber\\
    =&\int_{1/2+\epsilon-iT}^{1/2+\epsilon+iT}L^2(s,\ \chi)L^3(s,\ {\rm{sym}^2}f\otimes \chi)\times\nonumber\\
    &\qquad L(s,\ {\rm{sym}}^4f\otimes \chi)U(s)\frac{(x+1)^s}{s}\ \ud s\nonumber\\
    \ll& \int_{-T}^T \mid L^2(\frac{1}{2}+\epsilon+it,\ \chi)L^3(\frac{1}{2}+\epsilon+it,\ {\rm{sym}}^2f\otimes\chi)\times\nonumber\\
    &\qquad L(\frac{1}{2}+\epsilon+it,\ {\rm{sym}}^4f\otimes\chi)\mid \frac{x^{\frac{1}{2}+\epsilon}}{\mid 1/2+\epsilon+it\mid}\ \ud t\nonumber\\
    \ll& x^{\frac{1}{2}+\epsilon}+\sup_{1\leq T_1\leq T}x^{\frac{1}{2}+\epsilon} T_1^{-1}I_1^\frac{1}{2}I_2^\frac{1}{2},\label{Eq3.4}
\end{align}
where
\begin{align*}
    I_1=&\int_{T_1}^{2T_1}\mid L^2(1/2+\epsilon+it,\ \chi)\mid^2\ \ud t,\\
    I_2=&\int_{T_1}^{2T_1}\mid L^3(1/2+\epsilon+it,\ {\rm{sym}^2}f\otimes \chi)L(1/2+\epsilon+it,\ {\rm{sym}}^4f\otimes \chi)\mid^2\ \ud t.
\end{align*}
By \eqref{Eq2.14}, \eqref{Eq2.17} and \eqref{Eq2.22}, we have
\begin{align*}
    I_1\ll& (qT_1)^{2(1-\frac{1}{2})+\epsilon}\ll(qT_1)^{1+\epsilon},\\
    I_2\ll& \left(\max_{T_1\leq t\leq 2T_1}\mid L^3(1/2+\epsilon+it,\ {\rm{sym}}^2f\otimes \chi)\mid^2\right)\times\\&\qquad\left(\int_{T_1}^{2T_1}\mid  L(1/2+\epsilon+it,\ {\rm{sym}}^4f\otimes \chi)\mid^2\ \ud t\right)\\
    \ll&(qT_1)^{6\times \frac{67}{46}\times \frac{1}{2}+\epsilon}(qT_1)^{\frac{5}{2}+\epsilon}\\
    \ll&(qT_1)^{\frac{316}{46}+\epsilon}.
\end{align*}
By \eqref{Eq3.4}, we have
\begin{align*}
    J_3=&\int_{1/2+\epsilon-iT}^{1/2+\epsilon+iT}F_f(s,\ \chi)\frac{(x+1)^s}{s}\ \ud s\\
    \ll&x^{\frac{1}{2}+\epsilon}(qT)^{\frac{1}{2}+\frac{316}{92}+\epsilon}T^{-1}\\
    \ll&x^{\frac{1}{2}+\epsilon} q^{\frac{362}{92}+\epsilon}T^{\frac{270}{92}+\epsilon}.
\end{align*}
Now for horizontal portions, we have by \eqref{Eq2.13}, \eqref{Eq2.17} and \eqref{Eq2.21}
\begin{align*}
    J_4:=&\int_{1/2+\epsilon-iT}^{1+\epsilon+iT}F_f(s,\ \chi)\frac{(x+1)^s}{s}\ \ud s\\
    \ll& \int_{1/2+\epsilon}^{1+\epsilon}(qT)^{\frac{1}{3}\times 2(1-\sigma)+\frac{67}{46}(1-\sigma)\times 3+\frac{5}{2}(1-\sigma)}\frac{x^\sigma}{T}\ \ud\sigma\\
    \ll&\frac{(qT)^\frac{1040}{138}}{T}\int_{1/2+\epsilon}^{1+\epsilon}\left(\frac{x}{(qT)^{\frac{1040}{138}}}\right)^\sigma \ \ud\sigma\\
    \ll&\max_{\frac{1}{2}+\epsilon\leq \sigma\leq 1+\epsilon}\frac{(qT)^\frac{1040}{138}}{T}\left(\frac{x}{(qT)^\frac{1040}{138}}\right)^\sigma\\
    \ll&\frac{x^{1+\epsilon}}{T}+x^{\frac{1}{2}+\epsilon}q^{\frac{1040}{276}+\epsilon}T^{\frac{764}{276}+\epsilon}.
\end{align*}
Hence 
\begin{equation}
    \sum_{{\substack{n\leq x+1\\\chi\neq \chi_0\\\textit{$\chi$ is primitive}}}}\lambda_f^4(n)\chi(n)\ll \frac{x^{1+\epsilon}}{T}+x^{\frac{1}{2}+\epsilon}q^{\frac{362}{92}+\epsilon}T^{\frac{270}{92}+\epsilon}.\label{Eq3.5}
\end{equation}
For $\sum_3$, when $\chi$ is non-primitive character modulo $q$ there exists a conductor $q_1$ the smallest divisor  of $q$ such that $\chi=\chi_0\chi^*$, where $\chi^*$ is primitive character modulo $q_1$ and $\chi_0$ is principal character modulo $q$. By lemma 2.3 and by the same kind of arguments as in $\sum_2$, we have
\begin{equation}
    \sum_{{\substack{n\leq x+1\\\chi\neq \chi_0\\\textit{$\chi$ is non- primitive}}}}\lambda_f^4(n)\chi(n)\ll \frac{x^{1+\epsilon}}{T}+x^{\frac{1}{2}+\epsilon}q_1^{\frac{362}{92}+\epsilon}T^{\frac{270}{92}+\epsilon}.\label{Eq3.6}
\end{equation}
Now by writing  $L^3(1,\ {\rm{sym}}^2f\otimes \chi_0)L(1,\ {\rm{sym}}^4f\otimes \chi_0)U(1)=c_1$, from \eqref{Eq3.1}, \eqref{Eq3.3}, \eqref{Eq3.5} and \eqref{Eq3.6}, we have 
\begin{equation*}
    \sum_{n\leq x+1}\lambda_f^4(n)\chi(n)=c_1x\log x\frac{\phi(q)}{q^2}+ O\left(\frac{x^{1+\epsilon}}{T\phi(q)}\right)+O\left(\frac{x^{\frac{1}{2}+\epsilon}q^{\frac{362}{92}+\epsilon}T^{\frac{270}{92}+\epsilon}}{\phi(q)}\right).
\end{equation*}
Choosing $T=\frac{x^{\frac{23}{181}}}{q}$, we have
\begin{equation*}
    \sum_{n\leq x+1}\lambda_f^4(n)\chi(n)=c_1x\log x\frac{\phi(q)}{q^2}+ O\left(\frac{x^{\frac{158}{181}+\epsilon}q^{1+\epsilon}}{\phi(q)}\right),
\end{equation*}
for $q\ll x^{\frac{23}{181}-\epsilon}$, note that $\frac{x^{\frac{158}{181}}q}{\phi(q)}=o\left(x\log x\frac{\phi(q)}{q^2}\right)$ by lemma 2.11 for exceptionally large $x$, so that the main term dominates.\\
This proves theorem 1.
\end{proof}

\begin{proof}[Proof of theorem 2]
Observe that when $q$ is a prime, $\chi_0$ is the principle character modulo $q$ and rest of them are primitive and non-principle and there are $\phi(q)$ in characters.

Now for any Dirichlet character modulo $q$, where $q$ is prime, by orthogonality and by above observation, we have
\begin{align}
    \sum_{\substack{n\leq x+1\\n\equiv 1(q)}}\lambda_f^4(n)=&\frac{1}{\phi(q)}\sum_{\chi(q)}\sum_{n\leq x+1}\lambda_f^4(n)\chi(n)\nonumber\\
    =&\frac{1}{\phi(q)}\sum_{n\leq x+1}\lambda_f^4(n)\chi_0(n)+\frac{1}{\phi(q)}\sum_{{\substack{n\leq x+1,\\\chi\neq \chi_0}}}\lambda_f^4(n)\chi(n)\nonumber\\
    :=&\textit{$\sum$}_1+\textit{$\sum$}_2.\label{Eq4.1}
\end{align}
Now for $\sum_1$ and $\sum_2$ by the same arguments in the proof of theorem 1, we have
\begin{align}
    \sum_{n\leq x+1}\lambda_f^4(n)\chi_0(n)\nonumber\\=&x\log x\prod_{p\mid q}\left(1-\frac{1}{p}\right)^2L^3(1,\ {\rm{sym}}^2f\otimes\chi_0)\times \nonumber\\&L(1,\ {\rm{sym}}^4f\otimes \chi_0)U(1)+O\left(\frac{x^{1+\epsilon}}{T}\right)+O\left(x^{\frac{1}{2}
     +\epsilon}T^{\frac{111}{40}+\epsilon}\right)\label{Eq4.2}
\end{align}
and
\begin{equation*}
    \sum_{{\substack{n\leq x+1,\\\chi\neq \chi_0,\\\textit{$\chi$ is primitive}}}}\lambda_f^4(n)\chi(n)=-\frac{1}{2\pi i} \int_{C_2}F_f(s,\ \chi)\frac{(x+1)^s}{s}\ \ud s+O\left(\frac{x^{1+\epsilon}}{T}\right)
\end{equation*}
where $C_2$ is the curve joining the points $1+\epsilon+iT,\ \frac{1}{2}+\epsilon+iT,\ \frac{1}{2}+\epsilon-iT$ and $1+\epsilon-iT$ through straight line segments, $1\leq T\leq x$ is a parameter to be chosen later.\\
Now for vertical portion using \eqref{Eq2.1} and Cauchy-Schwarz inequality, we have  
\begin{align}
    J_5:=&\int_{1/2+\epsilon-iT}^{1/2+\epsilon+iT}F_f(s,\ \chi)\frac{(x+1)^s}{s}\ \ud s\nonumber\\
    =&\int_{1/2+\epsilon-iT}^{1/2+\epsilon+iT}L^2(s,\ \chi)L^3(s,\ {\rm{sym}^2}f\otimes \chi)\times\nonumber\\
    &\qquad L(s,\ {\rm{sym}}^4f\otimes \chi)U(s)\frac{(x+1)^s}{s}\ \ud s\nonumber\\
    \ll& \int_{-T}^T \mid L^2(\frac{1}{2}+\epsilon+it,\ \chi)L^3(\frac{1}{2}+\epsilon+it,\ {\rm{sym}}^2f\otimes\chi)\times\nonumber\\
    &\qquad L(\frac{1}{2}+\epsilon+it,\ {\rm{sym}}^4f\otimes\chi)\mid \frac{x^{\frac{1}{2}+\epsilon}}{\mid 1/2+\epsilon+it\mid}\ \ud t\nonumber\\
    \ll& x^{\frac{1}{2}+\epsilon}+\sup_{1\leq T_1\leq T}x^{\frac{1}{2}+\epsilon} T_1^{-1}I_1^\frac{1}{6}I_2^\frac{3}{4}I_3^\frac{1}{12},\label{Eq4.3}
\end{align}
where
\begin{align*}
    I_1=&\int_{T_1}^{2T_1}\mid L(1/2+\epsilon+it,\ \chi)\mid^{12}\ \ud t,\\
    I_2=&\int_{T_1}^{2T_1}\mid L(1/2+\epsilon+it,\ {\rm{sym}^2}f\otimes \chi)\mid^4\ \ud t,\\
    I_3=&\int_{T_1}^{2T_1}\mid L(1/2+\epsilon+it,\ {\rm{sym}}^4f\otimes \chi)\mid^{12} \ud t.
\end{align*}
For $I_1$ and $I_2$ using \eqref{Eq2.15} and \eqref{Eq2.18} respectively, and for $I_3$ using a similar kind of bound as in \eqref{Eq2.22}, we have
\begin{align*}
    I_1&\ll q^{4(1-\frac{1}{2}+\epsilon)}T_1^{3-2(1-\frac{1}{2}+\epsilon)+\epsilon}\ll q^{2+4\epsilon}T_1^{2+\epsilon}\\
    I_2&\ll (qT_1)^{6(1-\frac{1}{2}+\epsilon)}\ll (qT_1)^{3+6\epsilon}\\
    I_3&\ll (qT_1)^{30(1-\frac{1}{2}+\epsilon)}\ll(qT_1)^{15+30\epsilon}.
\end{align*}
From \eqref{Eq4.3}, we have
\begin{align}
    \int_{1/2+\epsilon-iT}^{1/2+\epsilon+iT}F_f(s,\ \chi)\frac{(x+1)^s}{s}\ \ud s&\ll x^{\frac{1}{2}+\epsilon}q^{\frac{1}{3}+\epsilon}T^{\frac{1}{3}-1+\epsilon}\left(qT\right)^{\frac{9}{4}+\frac{15}{12}+2\epsilon}\nonumber\\
    &\ll x^{\frac{1}{2}+\epsilon}q^{\frac{23}{6}+\epsilon}T^{\frac{17}{6}+\epsilon}.\label{Eq4.4}
\end{align}
Now for horizontal portions from theorem 1, we have
\begin{align}
    \int_{1/2+\epsilon-iT}^{1+\epsilon+iT}F_f(s,\ \chi)\frac{(x+1)^s}{s}\ \ud s&\ll \frac{x^{1+\epsilon}}{T}+x^{\frac{1}{2}+\epsilon}q^{\frac{1040}{276}+\epsilon}T^{\frac{764}{276}+\epsilon}.\label{Eq4.5}
\end{align}
From \eqref{Eq4.1}, \eqref{Eq4.4} and \eqref{Eq4.5} we have
\begin{equation*}
    \sum_{n\leq x+1}\lambda_f^4(n)\chi(n)=c_1x\log x\frac{\phi(q)}{q^2}+ O\left(\frac{x^{1+\epsilon}}{T\phi(q)}\right)+O\left(\frac{x^{\frac{1}{2}+\epsilon}q^{\frac{23}{6}+\epsilon}T^{\frac{17}{6}+\epsilon}}{\phi(q)}\right).
\end{equation*}
Choosing $T=\frac{x^{\frac{3}{23}}}{q}$, we have
\begin{equation*}
    \sum_{n\leq x+1}\lambda_f^4(n)\chi(n)=c_1x\log x\frac{\phi(q)}{q^2}+ O\left(\frac{x^{\frac{20}{23}+\epsilon}q^{1+\epsilon}}{\phi(q)}\right),
\end{equation*}
for $q\ll x^{\frac{3}{23}-\epsilon}$, note that $\frac{x^{\frac{20}{23}}q}{\phi(q)}=o\left(x\log x\frac{\phi(q)}{q^2}\right)$ by lemma 2.11 for exceptionally large $x$, so that the main term dominates.\\
This proves theorem 2.
\end{proof}

\begin{remark}
For general $q\geq 100$, in theorem 1, uniformity in $q$ gives upto $q\ll x^{\frac{23}{181}-\epsilon}$. However when $q\geq 100$ and $q$ is a prime integer, theorem 2 allows us the prime $q$-uniformity up to $q\ll x^{\frac{3}{23}-\epsilon}$ which is comparatively of larger range. We note that $\frac{3}{23}>\frac{23}{181}$.
\end{remark}

\begin{proof}[Proof of theorem 3]
By the definition of $k$-full kernel function $a(n)$, when we decompose $n\geq 1$ uniquely as $n=q(n)k(n),\ (k(n),\ q(n))=1$ where $k(n)$ is $k$-full and $q(n)$ is $k$-free, we have 
\begin{align*}
    a(n)=&a(k(n))\ll n^\epsilon,\ \qquad\lambda_f^4(n)\ll n^\epsilon.
\end{align*}
Let $1\leq H\leq x^{\frac{23}{543}}$ is a parameter to be chosen later. By \eqref{Eq2.25}, we have
\begin{align}
    \sum_{n\leq x}a(n)\lambda_f^4(n+1)=&\sum_{\substack{n\leq x\\k(n)\leq H}}a(n)\lambda_f^4(n+1)+\sum_{\substack{n\leq x\\k(n)> H}}a(n)\lambda_f^4(n+1)\nonumber\\
    =&\sum_{\substack{n\leq x\\k(n)\leq H}}a(n)\lambda_f^4(n+1)\nonumber\\
    &+O\left(\sum_{H<k(n)\leq x}a(k(n))\sum_{\substack{q(n)\leq \frac{x}{k(n)}\\(q(n),\ k(n))=1}}\lambda_f^4(k(n)q(n)+1)\right)\nonumber\\
    =&\sum_{\substack{n\leq x\\k(n)\leq H}}a(n)\lambda_f^4(n+1)\nonumber\\
&+O\left(\sum_{H<k(n)\leq x}k(n)^{2\epsilon}\sum_{\substack{q(n)\leq \frac{x}{k(n)}\\(q(n),\ k(n))=1}}q(n)^\epsilon\right)\nonumber\\
=&\sum_{\substack{n\leq x\\k(n)\leq H}}a(n)\lambda_f^4(n+1)+O\left(x^{1+\epsilon}\sum_{H<k(n)\leq x}\frac{1}{k(n)}\right)\nonumber\\
=&\sum_{\substack{n\leq x\\k(n)\leq H}}a(n)\lambda_f^4(n+1)+O\left(x^{1+\epsilon}H^{\frac{1}{k}-1}\right).\label{Eq5.1}
\end{align}
Now, if we define $\displaystyle{g(l)=\sum_{md^k=l}\mu(d)}$ then notice that $g(q(n))=1$. We have
\begin{align}
    \sum_{\substack{n\leq x\\k(n)\leq H}}a(n)\lambda_f^4(n+1)=&\sum_{k(n)\leq H}a(k(n))\sum_{\substack{q(n)\leq \frac{x}{k(n)}\\(q(n),\ k(n))=1}}\lambda_f^4\left(k(n)q(n)+1\right)\nonumber\\
    =&\sum_{k(n)\leq H}a(k(n))\sum_{\substack{q(n)\leq \frac{x}{k(n)}\\(q(n),\ k(n))=1}}g\left(q(n)\right)\lambda_f^4\left(k(n)q(n)+1\right)\nonumber\\
    =&\sum_{k(n)\leq H}a(k(n))\sum_{\substack{q(n)\leq \frac{x}{k(n)}\\(q(n),\ k(n))=1}}\left(\sum_{m(n)d^k(n)=q(n)}\mu(d(n))\right)\times\nonumber\\&\qquad\qquad
    \lambda_f^4\left(k(n)m(n)d^k(n)+1\right)\nonumber\\
    =&\sum_{k(n)\leq H}a(k(n))\sum_{\substack{d(n)\leq\left( \frac{x}{k(n)}\right)^{\frac{1}{k}}\\(d(n),\ k(n))=1}}\mu(d(n))\times\nonumber\\&\qquad\qquad\sum_{\substack{m(n)\leq \frac{x}{k(n)d^k(n)}\\(m(n),\ k(n))=1}}\lambda_f^4\left(k(n)m(n)d^k(n)+1\right)\nonumber\\
    :=&\textit{$\sum$}^*_1+\textit{$\sum$}^*_2,\label{Eq5.2}
\end{align}
where 
\begin{align*}
    \textit{$\sum$}^*_1=&\sum_{k(n)\leq H}a(k(n))\sum_{\substack{d(n)\leq H^{\frac{1}{k}}\\(d(n),\ k(n))=1}}\mu(d(n))\times\\
    &\qquad\qquad\sum_{\substack{m(n)\leq \frac{x}{k(n)d^k(n)}\\(m(n),\ k(n))=1}}\lambda_f^4\left(k(n)m(n)d^k(n)+1\right)\\
    \textit{$\sum$}^*_2=&\sum_{k(n)\leq H}a(k(n))\sum_{\substack{H^{\frac{1}{k}}<d(n)\leq \big( \frac{x}{k(n)}\big)^{\frac{1}{k}}\\(d(n),\ k(n))=1}}\mu(d(n))\times\\&\qquad\qquad\sum_{\substack{m(n)\leq \frac{x}{k(n)d^k(n)}\\(m(n),\ k(n))=1}}\lambda_f^4\left(k(n)m(n)d^k(n)+1\right).
\end{align*}
Now for $\sum_2^*$, we have
\begin{align*}
    \textit{$\sum$}_2^*\ll& x^{1+\epsilon}\sum_{k(n)\leq H}\frac{1}{k(n)}\sum_{d(n)\geq H^\frac{1}{k}}\frac{1}{d^k(n)}\\
    \ll&x^{1+\epsilon}H^{\frac{1}{k}-1}.
\end{align*}
Now for $\sum_1^*$ by using $\displaystyle{\sum_{\delta(n)/(m(n),\ k(n))}}\mu(\delta(n))=1$, we have
\begin{align*}
    \textit{$\sum$}^*_1=&\sum_{k(n)\leq H}a(k(n))\sum_{\substack{d(n)\leq H^{\frac{1}{k}}\\(d(n),\ k(n))=1}}\mu(d(n))\sum_{\delta(n)/k(n)}\mu(\delta(n))\times\\&\qquad\qquad\sum_{m_1(n)\delta(n)k(n)d^k(n)\leq x}\lambda_f^4\left(m_1(n)\delta(n)k(n)d^k(n)+1\right).
\end{align*}
Since we can write 
\begin{align*}
    \sum_{m_1(n)\delta(n)k(n)d^k(n)\leq x}\lambda_f^4\left(m_1(n)\delta(n)k(n)d^k(n)+1\right)\\=\sum_{\substack{n\leq x\\n\equiv1\ \left(\delta(n)k(n)d^k(n)\right)}}\lambda_f^4\left(n+1\right),
\end{align*}
by theorem 1, we have $\sum_1^*=\sum_1^{'}+\sum_1^{''}$,
where 
\begin{align*}
     \textit{$\sum$}_{1}^{'}=&\sum_{k(n)\leq H}a(k(n))\sum_{\substack{d(n)\leq H^{\frac{1}{k}}\\(d(n),\ k(n))=1}}\mu(d(n))\sum_{\delta(n)/k(n)}\mu(\delta(n))\times\\&\qquad\qquad c_1x\log x\frac{\phi\left(\delta(n)k(n)d^k(n)\right)}{\left(\delta(n)k(n)d^k(n)\right)^2}
\end{align*}
and 
\begin{align*}
     \textit{$\sum$}_{1}^{''}=&O\Bigg(\sum_{k(n)\leq H}a(k(n))\sum_{\substack{d(n)\leq H^{\frac{1}{k}}\\(d(n),\ k(n))=1}}\mu(d(n))\sum_{\delta(n)/k(n)}\mu(\delta(n))\times\\&\qquad\qquad x^{\frac{158}{181}+\epsilon} \frac{\left(\delta(n)k(n)d^k(n)\right)^{1+\epsilon}}{\phi\left(\delta(n)k(n)d^k(n)\right)}\Bigg).
\end{align*}
For $\sum_1^{''}$ by \eqref{Eq2.24}, we have
\begin{align}
   \textit{$\sum$}_{1}^{''}=&O\bigg(\sum_{k(n)\leq H}a(k(n))\sum_{\substack{d(n)\leq H^{\frac{1}{k}}\\(d(n),\ k(n))=1}}\mu(d(n))\sum_{\delta(n)/k(n)}\mu(\delta(n))\times\nonumber\\
   &\qquad\qquad x^{\frac{158}{181}+\epsilon} \frac{\left(\delta(n)k(n)\right)^{1+\epsilon}}{\phi\left(\delta(n)k(n)\right)}\frac{\left
   ( d^k(n)\right)^{1+\epsilon}}{\phi\left(d^k(n)\right)}\Bigg)\nonumber\\
   \ll&x^{\frac{158}{181}+2\epsilon}\sum_{k(n)\leq H}\sum_{\substack{d(n)\leq H^{\frac{1}{k}}\\(d(n),\ k(n))=1}}\sum_{\delta(n)/k(n)}\log \log \left(\delta(n)k(n)\right)\log \log \left(d^k(n)\right)\nonumber\\
   \ll&x^{\frac{158}{181}+2\epsilon} \log \log H^2\sum_{k(n)\leq H}\sum_{\substack{d(n)\leq H^{\frac{1}{k}}\\(d(n),\ k(n))=1}}\sum_{\delta(n)/k(n)}1\nonumber\\
   \ll& x^{\frac{158}{181}+10\epsilon} H^\frac{1}{k}H^\frac{1}{k}\nonumber\\
   \ll & x^{\frac{158}{181}+\epsilon} H^{\frac{2}{k}}.\label{Eq5.3}
\end{align}
Now for $\sum_1^{'}$, by \eqref{Eq2.25} we have
\begin{align}
     \textit{$\sum$}_1^{'}=&\sum_{k(n)\leq H}a(k(n))\sum_{\substack{d(n)\leq H^{\frac{1}{k}}\\(d(n),\ k(n))=1}}\mu(d(n))\sum_{\delta(n)/k(n)}\mu(\delta(n))\times\nonumber\\&\qquad\qquad c_1x\log x\frac{\phi\left(\delta(n)k(n)d^k(n)\right)}{\left(\delta(n)k(n)d^k(n)\right)^2}\nonumber\\
     =&c_1x\log x\sum_{k(n)\leq H}\frac{a(k(n))}{k(n)}\sum_{\substack{d(n)\leq H^{\frac{1}{k}}\\(d(n),\ k(n))=1}}\frac{\mu(d(n))}{d(n)}\sum_{\delta(n)/k(n)}\frac{\mu(\delta(n))}{\delta(n)}\times\nonumber\\&\qquad\qquad  \frac{\phi\left(\delta(n)k(n)d^k(n)\right)}{\delta(n)k(n)d^k(n)}\nonumber\\
     =&c_1x\log x\sum_{k(n)=1}^\infty\frac{a(k(n))}{k(n)}\sum_{\delta(n)/k(n)}\frac{\mu(\delta(n))}{\delta(n)}\sum_{\substack{d(n)\leq H^{\frac{1}{k}}\\(d(n),\ k(n))=1}}\frac{\mu(d(n))}{d(n)}\times \nonumber\\
     &\qquad\qquad \frac{\phi\left(\delta(n)k(n)d^k(n)\right)}{\delta(n)k(n)d^k(n)}\nonumber\\
     &+O\left(c_1x\log x\sum_{k(n)>H}\frac{a(k(n))}{k(n)}\sum_{\delta(n)/k(n)}\frac{\mu(\delta(n))}{\delta(n)}\sum_{\substack{d(n)\leq H^\frac{1}{k}\\(d(n),\ k(n))=1}}\frac{\mu(d(n))}{d(n)}\right)\nonumber\\
     =&c_1x\log x\sum_{k(n)=1}^\infty\frac{a(k(n))}{k(n)}\sum_{\delta(n)/k(n)}\frac{\mu(\delta(n))}{\delta(n)}\sum_{\substack{d(n)=1\\(d(n),\ k(n))=1}}^\infty\frac{\mu(d(n))}{d(n)}\times \nonumber\\&\qquad\qquad \frac{\phi\left(\delta(n)k(n)d^k(n)\right)}{\delta(n)k(n)d^k(n)}\nonumber\\
     &+O\left(c_1x\log x\sum_{k(n)=1}^\infty\frac{a(k(n))}{k(n)}\sum_{\delta(n)/k(n)}\frac{\mu(\delta(n))}{\delta(n)}\sum_{d(n)>H^\frac{1}{k}}\frac{\mu(d(n))}{d(n)}\right)\nonumber\\
     &+O\left(c_1x\log x\sum_{k(n)>H}\frac{a(k(n))}{k(n)}\sum_{\delta(n)/k(n)}\frac{\mu(\delta(n))}{\delta(n)}\sum_{\substack{d(n)\leq H^\frac{1}{k}\\(d(n),\ k(n))=1}}\frac{\mu(d(n))}{d(n)}\right)\nonumber\\
     =&c_2x\log x+O\left(x^{1+\epsilon}H^{\frac{1}{k}-1}\right).\label{Eq5.4}
\end{align}
From \eqref{Eq5.1}, \eqref{Eq5.2}, \eqref{Eq5.3} and \eqref{Eq5.4}, we have
\begin{equation*}
    \sum_{n\leq x}a(n)\lambda_f^4(n+1)=c_2x\log x+O\left(x^{1+\epsilon}H^{\frac{1}{k}-1}\right)+O\left(x^{\frac{158}{181}+\epsilon}H^\frac{2}{k}\right).
\end{equation*}
Note that by theorem 1, we have $q=\delta(n)k(n)d^k(n)$, and $\delta(n)\leq H,\ k(n)\leq H$ and $d^k(n)\leq H$. So we have $q\leq H^3$, but from theorem 1, $q$ has to satisfy the condition $q\ll x^{\frac{23}{181}-10\epsilon}$. This forces us to choose $H$ to be the optimal possible value satisfying $q\leq H^3$ and $q\ll x^{\frac{23}{181}-\epsilon}$. Thus we choose $H=x^{\frac{23}{543}-100\epsilon}$ and we have
\begin{equation*}
    \sum_{n\leq x}a(n)\lambda_f^4(n+1)=c_2x\log x+O\left(x^{\frac{520k+23}{543k}+\epsilon}\right).
\end{equation*}
This proves theorem 3.
\end{proof}

\textbf{Acknowledgements:} The first author wishes to express his gratitude to the Funding Agency "Ministry of Human Resource Development (MHRD), Govt. of India" for the fellowship PMRF, ID:3701831 for it's financial support. The authors are thankful to Prof. Saurabh kumar singh for drawing their attention to the paper \cite{9}.

\end{document}